\documentclass[12pt,reqno]{amsart}
\usepackage{fullpage}
\usepackage{times}
\usepackage{colonequals}
\usepackage{amsmath,amssymb,amsthm,url}
\usepackage[utf8]{inputenc}
\usepackage[english]{babel}
\usepackage{comment}
\usepackage{bbm}
\usepackage{enumerate}
\usepackage{bm}
\usepackage{graphicx}
\usepackage{mathrsfs}
\usepackage[colorlinks=true, pdfstartview=FitH, linkcolor=blue, citecolor=blue, urlcolor=blue]{hyperref}

\newtheorem{thm}{Theorem}[section]

\newtheorem{cor}[thm]{Corollary}

\theoremstyle{definition}

\newcommand{\F}{\mathbb{F}}

\title{A strengthening of McConnel's theorem on permutations over finite fields}

\author{Chi Hoi Yip}
\address{School of Mathematics\\ Georgia Institute of Technology\\ GA 30332\\ United States}
\email{cyip30@gatech.edu}
\keywords{finite field, linearized polynomial, direction}
\subjclass[2020]{11T06, 11B30}

\begin{document}

\begin{abstract}
Let $p$ be a prime, $q=p^n$, and $D \subset \F_q^*$. A celebrated result of McConnel states that if $D$ is a proper subgroup of $\F_q^*$, and $f:\F_q \to \F_q$ is a function such that $(f(x)-f(y))/(x-y) \in D$ whenever $x \neq y$, then $f(x)$ necessarily has the form $ax^{p^j}+b$. In this notes, we give a sufficient condition on $D$ to obtain the same conclusion on $f$. In particular, we show that McConnel's theorem extends if $D$ has small doubling.
\end{abstract}

\maketitle

\section{Introduction}
Throughout this paper, let $p$ be a prime. Let $q=p^n$ and $\F_q$ be the finite field with $q$ elements with $\F_q^*=\F_q \setminus \{0\}$. 

A celebrated result of McConnel \cite{M62} states that if $D$ is a proper subgroup of $\F_q^*$, and $f:\F_q \to \F_q$ is a function such that 
\begin{equation}\label{eq:f}
\frac{f(x)-f(y)}{x-y} \in D
\end{equation}
whenever $x,y \in \F_q$ with $x \neq y$, then there are $a,b \in \F_q$ and an integer $0 \leq j \leq n-1$, such that $f(x)=a x^{p^j}+b$ for all $x \in \F_q$. This result was first proved by Carlitz \cite{C60} when $q$ is odd and $D$ consists of squares in $\F_q^*$, that is, $D$ is the subgroup of index $2$. Carlitz's theorem and McConnel's theorem have various connections with finite geometry, graph theory, and group theory; we refer to a nice survey by Jones \cite[Section 9]{J20}. In particular, they have many applications in finite geometry; see for example \cite{BSW92, KS10}. We also refer to variations of McConnel's theorem in \cite{BL73, G81, L90} via tools from group theory. 


One may wonder if the assumption that $D$ is a multiplicative subgroup plays an important role in McConnel's theorem and if it is possible to weaken this assumption. Inspired by this natural question, in this paper, we find a sufficient condition on $D$ so that if condition~\eqref{eq:f} holds for a function $f:\F_q \to \F_q$, then $f$ necessarily has the form $f(x)=a x^{p^j}+b$. In particular, our main result (Theorem~\ref{thm:main}) strengthens McConnel's theorem. 

Before stating our results, we introduce some motivations and backgrounds from the theory of directions and their applications.  Let $AG(2,q)$ denote the {\em affine Galois plane} over the finite field $\F_q$. Let $U$ be a subset of points in $AG(2,q)$; we use Cartesian coordinates in $AG(2,q)$ so that $U=\{(x_i,y_i):1 \leq i \leq |U|\}$.
The set of {\em directions determined by} $U \subset AG(2, q)$ is 
\[ \mathcal{D}_U=\left\{ \frac{y_j-y_i}{x_j-x_i} \colon 1\leq i <j \leq |U| \right \} \subset \F_q \cup \{\infty\},\]
 where $\infty$ is the vertical direction. If $f:\F_q \to \F_q$ is a function, we can naturally consider its graph $U(f)=\{(x,f(x)): x \in \F_q\}$ and the set of {\em directions determined by $f$} is $\mathcal{D}_f:=\mathcal{D}_{U(f)}$. Indeed, $\mathcal{D}_f$ precisely computes the set of slopes of tangent lines joining two points on $U(f)$. Using this terminology, condition~\eqref{eq:f} is equivalent to $\mathcal{D}_f \subset D$. 
 
The following well-known result is due to Blokhuis, Ball, Brouwer, Storme, and Sz{\H{o}}nyi \cite{B03, BBBSS99}.

\begin{thm}\label{thm:directions}
Let $p$ be a prime and let $q=p^n$. Let $f:\F_q \to \F_q$ be a function such that $f(0)=0$. If $|\mathcal{D}_f|\leq \frac{q+1}{2}$, then $f$ is a linearized polynomial, that is, there are $\alpha_0,\alpha_1, \ldots, \alpha_{n-1}\in \F_q$, such that 
$$
f(x)=\sum_{j=0}^{n-1} \alpha_j x^{p^j}, \quad \forall x \in \F_q.
$$
\end{thm}

In particular, when $q=p$ is a prime, Theorem~\ref{thm:directions} implies McConnel's theorem; this was first observed by Lov\'{a}sz and Schrijver \cite{LS83}. We remark that when $q=p$ is a prime, M\"uller \cite{M05} showed a stronger result, namely, if $D$ is a non-empty proper subset of $\F_p^*$ such that $f(x)-f(y)\in D$ whenever $x-y \in D$, then there are $a,b \in \F_p$ such that $f(x)=ax+b$ for all $x \in \F_p$. For a general prime power $q$, Theorem~\ref{thm:directions} does not imply McConnel's theorem directly. However, using the idea of directions, Muzychuk \cite{M20} provided a self-contained proof of McConnel's theorem.

While Theorem~\ref{thm:directions} is already quite powerful, in many applications, if some extra information on $\mathcal{D}_f$ is given, it is desirable to obtain a stronger conclusion, namely, $f(x)$ is of the form $ax^{p^j}+b$ for some $j$, or even $f(x)$ must have the form $ax+b$. As an illustration, we mention two recent works in applying a stronger version of Theorem~\ref{thm:directions} to prove analogues of the Erd\H{o}s-Ko-Rado (EKR) theorem in the finite field setting \cite{ACW24, AY22}. Asgarli and Yip \cite{AY22} proved the EKR theorem for a family of pseudo-Paley graphs of square order, and their main result roughly states that if $\mathcal{D}_f$ arises from a Cayley graph with ``nice multiplicative properties" on its connection set, then $f(x)$ has the form $ax+b$; see \cite{Y24} for more discussions. As another example, very recently, Aguglia, Csajb\'{o}k, and Weiner proved several EKR theorems for polynomials over finite fields \cite{ACW24}. Again, one key ingredient in their proof is a strengthening of Theorem~\ref{thm:directions}. In \cite[Theorem 2.2]{ACW24}, they showed that if $\mathcal{D}_f$ is a proper $\F_p$-subspace of $\F_q$, then $f(x)$ has the form $ax+b$; in \cite[Theorem 2.13]{ACW24} (see also \cite{GM14}), they showed that if $\mathcal{D}_f\setminus \{0\}$ is a subset of a coset of $K$, where $K$ is the subgroup of $\F_q^*$ with index $2$, then $f(x)$ has the form $ax^{p^j}+b$ for some $j$.

Inspired by the connection above between the theory of directions and McConnel's theorem, we establish the following result. 

\begin{thm}\label{thm:main}
Let $p$ be a prime. Let $q=p^n$ and $D \subset \F_q^*$. Suppose $f:\F_q \to \F_q$ is a function such that 
$$
\frac{f(x)-f(y)}{x-y} \in D
$$
whenever $x,y \in \F_q$ with $x \neq y$. If 
\begin{equation}\label{eq:bound}
|DD^{-1}D^{-1}|\leq \frac{q+1}{2},    
\end{equation}
then there are $a,b \in \F_q$ and an integer $0 \leq j \leq n-1$, such that $f(x)=a x^{p^j}+b$ for all $x \in \F_q$. In particular, if $c>0$ is a real number such that $|DD|\leq c|D|$, and $c^3|D|\leq \frac{q+1}{2}$, then the same conclusion holds.
\end{thm}

If $D$ is a proper subgroup of $\F_q^*$, then clearly $DD^{-1}D^{-1}=D$ and thus Theorem~\ref{thm:main} recovers McConnel's theorem. Indeed, in this case, the doubling constant of $D$ is simply $|DD|/|D|=1$. Our result shows that if the doubling constant $|DD|/|D|$ of $D$ is small, and $|D|$ is not too large, then the analogue of McConnel's theorem still holds. There are many ways to construct a set $D \subset \F_q^*$ with small doubling. For example, we can set $D=K \cup E$, where $K$ is a subgroup of $\F_q^*$, and $E$ is an arbitrary subset of $\F_q^*$ such that $|E|$ is small; alternatively, $D$ can be taken to be the union of some cosets of a fixed subgroup $K$ of $\F_q^*$.

For a linearized polynomial $f(x)$, note that $\mathcal{D}_f=\operatorname{Im}(f(x)/x)$. We refer to \cite{CMP19} and references therein on the study of $\operatorname{Im}(f(x)/x)$ for linearized polynomials $f$. In particular, it is an open question to determine all the possible sizes of $\operatorname{Im}(f(x)/x)$ among linearized polynomials $f$ \cite[Section 6]{CMP19}. Theorem~\ref{thm:main} implies the following corollary, which partially addresses this question. It states that if $D$ is such an image set and $D$ satisfies inequality~\eqref{eq:bound}, then $D$ is necessarily a coset of a subgroup of $\F_{q}^*$ with a restricted index.

\begin{cor}\label{cor}
Let $p$ be a prime and let $q=p^n$. If $D \subset \F_q^*$ and $|DD^{-1}D^{-1}|\leq \frac{q+1}{2}$, then $D=\mathcal{D}_f$ for some function $f: \F_q \to \F_q$ with $f(0)=0$ if and only if $D=aK$, where $a \in \F_q^*$ and $K$ is a subgroup of $\F_{q}^*$ with index $p^r-1$, where $r$ is a divisor of $n$.
\end{cor}

\textbf{Notations.} We follow standard notations for arithmetic operations among sets.  Given two sets $A$ and $B$, we write $AB=\{ab: a \in A, b \in B\}$, $A^{-1}=\{a^{-1}: a \in A\}$.

\section{Proofs}

We start by proving Theorem~\ref{thm:main}. Our proof is inspired by several arguments used in \cite{LS83, M20}.

\begin{proof}[Proof of Theorem~\ref{thm:main}]
Without loss of generality, by replacing the function $f(x)$ with $f(x)-f(0)$, we may assume that $f(0)=0$. Since $|\mathcal{D}_f|\leq \frac{q+1}{2}$, Theorem~\ref{thm:directions} implies that $f$ is linearized. Let $g(x)=1/f(x^{-1})$ for $x \in \F_q \setminus \{0\}$ and set $g(0)=0$. We claim that $g$ is also linearized.

Let $x,y \in \F_q^*$ with $x \neq y$. Then 
$$
\frac{g(x)-g(y)}{x-y}=\frac{\frac{1}{f(x^{-1})}-\frac{1}{f(y^{-1})}}{x-y}=\frac{f(y^{-1})-f(x^{-1})}{(x-y)f(x^{-1})f(y^{-1})}=\frac{f(y^{-1}-x^{-1})}{(x-y)f(x^{-1})f(y^{-1})}
$$
since $f$ is linearized. It follows that
$$
\frac{g(x)-g(y)}{x-y}=\frac{f(\frac{x-y}{xy})}{(x-y)f(x^{-1})f(y^{-1})}=\frac{f(\frac{x-y}{xy})}{\frac{x-y}{xy}} \cdot \frac{x^{-1}}{f(x^{-1})} \cdot \frac{y^{-1}}{f(y^{-1})} \in DD^{-1}D^{-1}.
$$
On the other hand, if $x \in \F_q^*$, then 
$$
\frac{g(x)-g(0)}{x-0}=\frac{g(x)}{x}=\frac{1}{xf(x^{-1})}=\frac{x^{-1}}{f(x^{-1})} \in D^{-1}.
$$
We conclude that 
$$
\mathcal{D}_g \subset DD^{-1}D^{-1} \cup D^{-1}=DD^{-1}D^{-1}.
$$
Then inequality~\eqref{eq:bound} implies that $|\mathcal{D}_g|\leq \frac{q+1}{2}$, and Theorem~\ref{thm:directions} implies that 
$g$ is also linearized. 

Since $f$ and $g$ are both linearized, we can find $\alpha_0, \beta_0, \ldots, \alpha_{n-1}, \beta_{n-1} \in \F_q$, such that
$$
f(x)=\sum_{j=0}^{n-1} \alpha_j x^{p^j}, \quad \text{and} \quad g(x)=\sum_{j=0}^{n-1} \beta_j x^{p^j}, \quad \forall x \in \F_q.
$$
By the definition of $g$, we have $f(x^{-1})g(x)=1$ for each $x \in \F_q^*$, and thus $(x^{p^{n-1}}f(x^{-1}))(g(x)/x)=x^{p^{n-1}-1}$ for all $x \in \F_q^*$. Equivalently, 
\begin{equation}\label{eq:equal}
h(x):=\bigg(\sum_{j=0}^{n-1} \alpha_j x^{p^{n-1}-p^j}\bigg) \bigg(\sum_{j=0}^{n-1} \beta_j x^{p^j-1} \bigg)=x^{p^{n-1}-1}
\end{equation}
holds for all $x \in \F_q^*$. Note that $h(x)-x^{p^{n-1}-1}$ is a polynomial with degree at most $2(p^{n-1}-1)\leq p^n-1=q-1$, and $h(x)-x^{p^{n-1}-1}$ vanishes on $\F_q^*$. It follows that there is a constant $C \in \F_q$, such that $h(x)-x^{p^{n-1}-1}=C(x^{q-1}-1)$ as polynomials. By setting $x=0$, we get $C=0$. Therefore, $h(x)=x^{p^{n-1}-1}$ as polynomials. In particular, $g(x)$ is a factor of $x^{p^{n-1}}$. Thus, there are $\gamma \in \F_q^*$ and $0 \leq j \leq n-1$ such that $g(x)=\gamma x^{p^j}$ for all $x \in \F_q$, and it follows that $f(x)=\gamma^{-1} x^{p^j}$ for all $x \in \F_q$, as required.

Finally, assume that $|DD|\leq c|D|$, and $c^3|D|\leq \frac{q+1}{2}$. Then the Pl\"unnecke–Ruzsa inequality (see for example \cite[Theorem 1.2]{P12}) implies that 
$$
|DD^{-1}D^{-1}|\leq c^3|D| \leq \frac{q+1}{2},
$$
and thus the same conclusion holds.
\end{proof}

Now we use Theorem~\ref{thm:main} to deduce Corollary~\ref{cor}.
\begin{proof}[Proof of Corollary~\ref{cor}]
Assume that $f:\F_q \to \F_q$ is a function such that $f(0)=0$ and $\mathcal{D}_f=D$. Since $D \subset \F_q^*$ and $|DD^{-1}D^{-1}|\leq \frac{q+1}{2}$, Theorem~\ref{thm:main} implies that there exist $a \in \F_q^*$ and an integer $0 \leq j \leq n-1$, such that $f(x)=ax^{p^j}$ for all $x \in \F_q$. Therefore, $D=\mathcal{D}_f=\{ax^{p^j-1}: x \in \F_q^*\}.$ Note that $$\gcd(p^j-1, q-1)=\gcd(p^j-1, p^n-1)=p^{\gcd(j,n)}-1.$$ It follows that $$D=\{ax^{p^{\gcd(j,n)}-1}: x \in \F_q^*\}=aK,$$ where $K$ is the subgroup of $\F_q^*$ with index $p^{\gcd(j,n)}-1$.     

Conversely, let $r$ be a divisor of $n$ and $a \in \F_q^*$. Let $f(x)=ax^{p^r}$ for all $x \in \F_q$. Then we have $\mathcal{D}_f=\{ax^{p^r-1}: x \in \F_q^*\}=aK$, where $K$ is the subgroup of $\F_q^*$ with index $p^{r}-1$.     
\end{proof}

\section*{Acknowledgements}
The author thanks Bence Csajb\'{o}k for helpful discussions.  The author is also grateful to anonymous referees for their valuable comments and suggestions. 

\bibliographystyle{abbrv}
\bibliography{main}

\end{document}